\documentclass[11pt,reqno,oneside]{amsart}

\usepackage{graphicx,subfigure,threeparttable}
\usepackage{amssymb}
\usepackage{epstopdf}
\usepackage[a4paper, total={6in, 9in}]{geometry}
\usepackage{algorithm}
\usepackage{algpseudocode}

\usepackage{booktabs} 
\usepackage{subfig} 

\usepackage{amsfonts}

\usepackage{amsmath,amsthm,mathrsfs,amssymb,cite}
\usepackage[usenames]{color}

\newtheorem{thm}{Theorem}[section]

\newtheorem{lem}{Lemma}[section]

\theoremstyle{definition}
\newtheorem{definition}{Definition}[section]

\newtheorem{problem}{Problem}

\theoremstyle{remark}

\newtheorem{rem}{Remark}[section]
\numberwithin{equation}{section}

\numberwithin{equation}{section}

\newcounter{saveeqn}

\newcommand{\eqnref}[1]{(\ref {#1})}

\title[Simultaneous recovery of flux and thickness]{\bf Simultaneous recovery of surface heat flux and thickness of a solid structure by ultrasonic measurements }

\author{Youjun Deng}
\address{School of Mathematics and Statistics, Central South University, Changsha, Hunan, China}
\email{youjundeng@csu.edu.cn; dengyijun\_001@163.com}

\author{Hongyu Liu}
\address{Department of Mathematics, City University of Hong Kong, Kowloon, Hong Kong, China}
\email{hongyu.liuip@gmail.com, hongyliu@cityu.edu.hk}

\author{Xianchao Wang}
\address{Department of Mathematics, City University of Hong Kong, Kowloon, Hong Kong, China}
\email{xcwang90@gmail.com}

\author{Dong Wei}
\address{State Key Laboratory of Aerodynamics, China Aerodynamics Research and Development Center,Mianyang,Sichuan,621000,China\\
Center of Nondestructive Examination, China Special Equipment Inspection and Research Institute, Beijing, 100029, China}
\email{xisuzisi@126.com}

\author{Liyan Zhu}
\address{School of Mathematics and Statistics, Central South University, Changsha, Hunan, China}

\date{} 
\begin{document}
\maketitle

\begin{abstract}

This paper is concerned with a practical inverse problem of simultaneously reconstructing the surface heat flux and the thickness of a solid structure from the associated ultrasonic measurements. In a thermoacoustic coupling model, the thermal boundary condition and the thickness of a solid structure are both unknown, while the measurements of the propagation time by ultrasonic sensors are given. We reformulate the inverse problem as a PDE-constrained optimization problem by constructing a proper objective functional. We then develop an alternating iteration scheme which combines the conjugate gradient method and the deepest decent method to solve the optimization problem. Rigorous convergence analysis is provided for the proposed numerical scheme. By using experimental real data from the lab, we conduct extensive numerical experiments to verify several promising features of the newly developed method.

\medskip

\medskip



\end{abstract}

\section{Introduction}\label{sect:1}


 The heat conduction is a ubiquitous phenomenon which forms the basis for many practical applications. Given the geometrical and material configurations of a material structure as well as the heat source including the initial and boundary temperature distributions, finding the temperature distribution as well as its change on the material structure constitutes the so-called direct or forward heat conduction problem. In many practical applications, one is more interested in the so-called inverse heat conduction problem which reverses the above forward problem through direct or indirect measurement data; see e.g. \cite{MSS, SSWBH, HGF, HHSW,MKPR, AOM, OMWYZ, YDY, WYYX, RBSK, FHLZ,AHLC} and the references cited therein for some related studies in the literature.

In this paper, motivated by practical applications (with experimental real data from the lab), we are mainly concerned with the reconstruction of the surface heat flux and the thickness of a solid structure by using the associated ultrasonic measurements. The reconstruction of the surface heat flux is one of the most typical inverse heat conduction problems, and is widely encountered in aerospace, nuclear physics, metallurgy, and other industrial fields; see \cite{HHSW,MKPR} and the references cited therein for more related discussions. Ultrasonic thickness measurement is a commonly used non-destructive testing method, and is widely used in petroleum, machinery, ship, chemical industry and other fields \cite{RBSK,FHLZ}. For most of existing results in the literature, one either recovers the surface heat flux by assuming the thickness of the material structure is a priori known, or recovers the thickness of the material structure by assuming the surface heat flux is a priori known. However, it is a more practical scenario that both the surface heat flux and the thickness of the material structure are unknown and one recovers both of them.

In this paper, based on the ultrasonic echo method and the inverse analysis method of the heat conduction, we propose a novel scheme for simultaneously recovering the surface heat flux and the thickness of the material structure through the pulse-echo measurements by the ultrasonic probe. The study is posed as an inverse problem associated with a thermoacoustic coupling model. We recast the inverse problem as a PDE (partial differential equation)-constrained optimization problem by constructing a proper objective functional. We then develop an alternating iteration scheme which combines the conjugate gradient method and the deepest decent method to solve the optimization problem. Rigours convergence analysis is provided for the proposed numerical scheme. Finally, by using experimental real data from the lab, we conduct extensive numerical experiments to verify effectiveness and efficiency of the method.

The rest of this paper is organized as follows. In Section 2, we present the mathematical formulation of the direct and inverse problems for our study and also briefly discuss the corresponding physical setup. In Section 3, we give the optimization formulation of the inverse problem and then derive the alternating iteration scheme for solving the optimization problem. Sections 4 and 5 are, respectively, devoted to the theoretical convergence analysis and numerical experiments.

\section{Mathematical and physical setups}

The physical principle of the ultrasonic thickness measurement is to use the propagation time of the ultrasonic waves in the medium to infer the thickness of the underlying solid structure. The propagation time is mainly determined by the thickness, material properties, and internal temperature field of the solid structure; see Figure~\ref{fig.Model} for a schematic illustration of the physical setup. The propagation time of the ultrasonic wave in the solid structure can be expressed as (see \cite{WSSG17}):
\begin{figure}[h]
\centering
\includegraphics[width=0.7\textwidth]{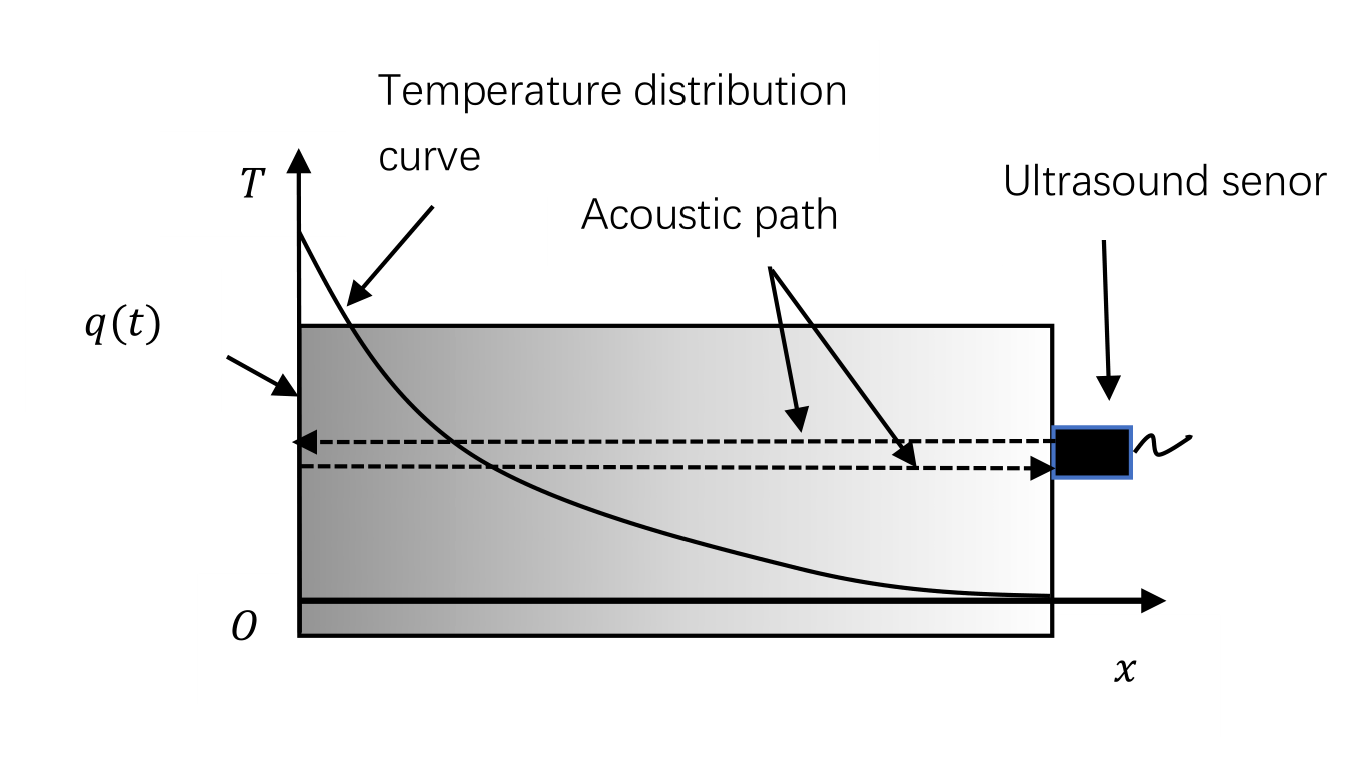}
\caption{ A one-dimensional model based on ultrasonic detection.} \label{fig.Model}
\end{figure}
\begin{equation}\label{eq:L1}
\Lambda_L(t)=2\int_{0}^{L}\frac{1}{V(T(x,t))}\,\mathrm{d}x,\quad t\in(0,+\infty),
\end{equation}
where $L\in \mathbb{R}_{+}$ denotes the unidirectional propagation distance of the ultrasonic wave in the medium, i.e. the thickness of the material structure being under detection. Here, $V$ is the propagation velocity of the acoustic wave in the solid medium and is related to the material properties and the structure temperature. Usually, it has an approximately linear relationship with the temperature, i.e.,
\begin{equation}\label{eq:velocity}
V(T)=aT+b, \quad a,\, b \in \mathbb{R},
\end{equation}
where $a$ and $b$ are determined by the properties of the material and calibrated by experiments. $T(x,t)$ is the internal temperature of the structure, which satisfies the following heat conduction system for $T\in W_2^{2,1}\left(\left[0,L \right]\times(0, \tau)\right)$:
\begin{equation}\label{eq:master}
\left\{
\begin{split}
&\rho c\frac{\partial T}{\partial t}=\frac{\partial }{\partial x}\left ( k\frac{\partial T}{\partial x} \right),\qquad\qquad \qquad\qquad \ \, (x,t)\in [0,L]\times[0,\tau],\\
&-k\frac{\partial T}{\partial x}\Big|_{x=0}=q, \quad -k\frac{\partial T}{\partial x}\Big|_{x=L}=0,\qquad\quad q(t)\in L^{2}\left[0,\tau \right],\\
&T\mid_{t=0}=T_{0},\qquad\qquad\qquad\qquad\qquad\qquad\quad x\in [0,L],
\end{split}
\right.
\end{equation}
where $k(x,t)$, $c(x,t)$ and $\rho(x)$ are the thermal conductivity, specific heat capacity and density of the material, respectively, and $q(t)$ denotes the heat flux density on the boundary.

In this paper, the inverse problem that we are concerned with is described as follows:
\begin{problem}\label{problem}
Given the measured propagation time of the ultrasonic wave $\Lambda_m(t)$ and the measured boundary temperature $T_m(L,t)$, determine the surface heat flux $q(t)$ and the thickness $L$, i.e.,
\begin{equation}\label{eq:ip1}
\left\{\Lambda_{m}\left(t\right),T_m(L,t)\right\}\rightarrow\{q\left(t\right),L\},\quad t\in \left[0,\tau\right].
\end{equation}
\end{problem}

We would like to point out that in the measured data, the temperature at the end of the solid structure can be measured. However the thickness $L$ of the structure is unknown in the practical application of our interest. It can be directly verified that the inverse problem \eqref{eq:ip1} is nonlinear.


%
\section{An alternating iteration scheme for the inverse problem}

In this section, we first recast the inverse problem \eqref{eq:ip1} as an optimization problem by following the general framework of Tikhonov regularization for inverse problems; see e.g. \cite{DYL} and the references cited therein. Then we present the newly proposed alternating iteration scheme. To that end, we introduce the following objective functional with respect to the unknown heat flux $q(t)$ and thickness $L$:
\begin{equation}\label{eq:functional}
J\left(q,L\right)=\frac{1}{2}\int_{0}^{\tau}\left(\Lambda_L\left(t\right)-\Lambda_m\left(t\right) \right)^2\mathrm{d}t +\frac{\alpha}{2}\int_{0}^{\tau}\left(T(L,t)-T_m(L,t)\right)^2\mathrm{d}t,
\end{equation}
where $\alpha\in\mathbb{R}_+$ signifies a regularization parameter. We then recast the inverse problem \eqref{eq:ip1} as the following PDE-constrained optimization problem:
\begin{equation}\label{eq:op1}
\min_{q\in L^2(0, \tau),\, L\in\mathbb{R}_+} J\left(q, L\right)\ \ \mbox{subject to} \ T\in W_2^{2,1}\left(\left[0,L \right]\times(0, \tau)\right)  \  \mbox{satisfying} \quad \eqnref{eq:master}.
\end{equation}
We next convert the constrained optimization problem into an unconstrained one by using the Lagrange multiplier method. Using the heat conduction equation \eqref{eq:master}, with the boundary conditions and the initial condition, the augmented functional is given as follows:
\begin{equation}\label{eq:lagrangian}
\begin{aligned}
J\left(q,L\right)=&\frac{1}{2} \int_{0}^{\tau}\left(\Lambda_L\left(t\right)-\Lambda_m\left(t\right) \right)^2\mathrm{d}t+\frac{\alpha}{2} \int_{0}^{\tau}\left(T(L,t)-T_m(L,t)\right)^2\mathrm{d}t\\
&- \int_{0}^{L} \int_{0}^{\tau}\left\{\rho c\frac{\partial T(x,t)}{\partial t}-\frac{\partial }{\partial x}\left ( k\frac{\partial T(x,t)}{\partial x} \right )\right\} \lambda_1(x,t)
\,\mathrm{d}t \mathrm{d}x\\
&+ \int_{0}^{\tau}\left( k\frac{\partial T(x,t)}{\partial x}+q(t) \right)\bigg |_{x=0} \lambda_2(t)\, \mathrm{d}t
+\int_{0}^{\tau}\left(k\frac{\partial T(x,t)}{\partial x}\right) \bigg |_{x=L} \lambda_3(t)\, \mathrm{d}t\\
&-\int_{0}^{L} \bigg(T(x,t)-T_0(x,t)\bigg)\bigg|_{t=0} \ \lambda_4(x) \, \mathrm{d}x.
\end{aligned}
\end{equation}
where $\lambda_1(x,t)$, $\lambda_2(t)$, $\lambda_3(t)$ and $\lambda_4(x)$ denote the Lagrange multipliers.

\subsection{Gradients with respect to the parameters}
To solve the optimization problem \eqref{eq:lagrangian}, the gradients of the objective functional $J$ with respective to $q$ and $L$ are required. However,
it is difficult to solve the gradients directly. Thus, we refer to \cite{ARK} and use the adjoint state method to derive them. Noting that
\begin{equation*}
\begin{aligned}
&\int_{0}^{L} \int_{0}^{\tau} \rho c \frac{\partial T(x,t)}{\partial t} \lambda_1(x,t) \,\mathrm{d}t \mathrm{d}x\\
&\quad =-\int_{0}^{L} \int_{0}^{\tau} \rho c \frac{\partial \lambda_1(x,t)}{\partial t} T(x,t) \,\mathrm{d}t \mathrm{d}x
+\int_{0}^{L} \bigg(\rho c\lambda_1(x,t) T\left(x,t\right)\bigg)\bigg |_{0}^{\tau}\,\mathrm{d}x.
\end{aligned}
\end{equation*}
Similarly, one can deduce that
\begin{equation*}
\begin{aligned}
&\int_{0}^{L} \int_{0}^{\tau} \frac{\partial }{\partial x}\left ( k\frac{\partial T(x,t)}{\partial x} \right ) \lambda_1(x,t)
\,\mathrm{d}t \mathrm{d}x
=\int_{0}^{L} \int_{0}^{\tau} \frac{\partial }{\partial x}\left ( k\frac{\partial \lambda_1(x,t)}{\partial x} \right ) T(x,t) \,\mathrm{d}t \mathrm{d}x\\
&\qquad\qquad\qquad+\int_{0}^{\tau} \left(k\lambda_1(x,t)\frac{\partial T(x,t)}{\partial x}\right)\bigg |_{0}^L\,\mathrm{d}t\quad - \int_{0}^{\tau} \left(k\frac{\partial \lambda_1(x,t)}{\partial x}T(x,t)\right)\bigg |_{0}^L\,\mathrm{d}t.
\end{aligned}
\end{equation*}
Finally, the equation \eqref{eq:lagrangian} can be rewritten as
\begin{equation*}\label{eq:lagrangian2}
\begin{aligned}
J(q,L)=&\frac{1}{2}\int_{0}^{\tau}\left(\Lambda_L\left(t\right)-\Lambda_m\left(t\right) \right)^2\mathrm{d}t+\frac{\alpha}{2} \int_{0}^{\tau}\left(T(L,t)-T_m(L,t)\right)^2\mathrm{d}t\\
&+\int_{0}^{L} \int_{0}^{\tau}\left\{\rho c\frac{\partial \lambda_1(x,t)}{\partial t}+ \frac{\partial }{\partial x}\left ( k\frac{\partial \lambda_1(x,t)}{\partial x} \right )\right\} T(x,t)
\,\mathrm{d}t \mathrm{d}x\\
&+\int_{0}^{\tau} \left(k\lambda_1(x,t)\frac{\partial T(x,t)}{\partial x}\right)\bigg |_{0}^L\,\mathrm{d}t
-\int_{0}^{\tau} \left(k\frac{\partial \lambda_1(x,t)}{\partial x}T(x,t)\right)\bigg |_{0}^L\,\mathrm{d}t\\
&-\int_{0}^{L} \bigg(\rho c\lambda_1(x,t) T(x,t)\bigg)\bigg |_{0}^{\tau}\,\mathrm{d}x+ \int_{0}^{\tau}\left( k\frac{\partial T(x,t)}{\partial x}+q(t) \right)\bigg |_{x=0} \lambda_2(t)\, \mathrm{d}t\\
&+\int_{0}^{\tau}\left(k\frac{\partial T(x,t)}{\partial x}\right) \bigg |_{x=L} \lambda_3(t)\, \mathrm{d}t-\int_{0}^{L} \bigg(T(x,t)-T_0(x,t)\bigg)\bigg|_{t=0} \ \lambda_4(x) \, \mathrm{d}x.
\end{aligned}
\end{equation*}

To obtain the adjoint state equation, we set
\begin{equation*}
\frac{\partial J}{\partial T}=0,
\end{equation*}
and it yields
\begin{equation*}
\left\{
\begin{aligned}
& \rho c\frac{\partial \lambda_1(x,t)}{\partial t}
+ \frac{\partial }{\partial x}\left ( k\frac{\partial \lambda_1(x,t)}{\partial x} \right )=S(x,t),\\
& -k\frac{\partial \lambda_1(x,t)}{\partial{x}}\bigg|_{x=0}=0,
\quad -k\frac{\partial{\lambda_1(x,t)}}{\partial{x}}\bigg|_{x=L}=0,\\
& \lambda_1(x,\tau)=0,\\
& \lambda_2(t)=\lambda_1(0,t),\\
& \lambda_3(t)=-\lambda_1(L,t),\\
& \lambda_4(x)=-\rho c \lambda_1(x,0).
\end{aligned}
\right.
\end{equation*}
Here the source term is given by
\begin{equation*}
S(x,t)=2\left(\Lambda_L\left(t\right)-\Lambda_m\left(t\right)\right)\frac{a}{(V(x,t))^2}+\frac{\alpha}{L}\left(T(L,t)-T_m(L,t)\right) .
\end{equation*}
From \eqref{eq:lagrangian}, through a straightforward calculation, the gradients with respect to the model parameters are given by
\begin{equation}\label{eq:grad1}
\left\{
\begin{aligned}
&\frac{\partial J}{\partial q}(t)=\lambda_2(t)=\lambda_1(0,t),\\
&\frac{\partial J}{\partial L}=\int_{0}^{\tau} \left(\Lambda_L\left(t\right)-\Lambda_m\left(t\right)\right)\frac{\partial \Lambda_L\left(t\right)}{\partial L}\, \mathrm{d}t
= \int_{0}^{\tau} \frac{2\left(\Lambda_L\left(t\right)-\Lambda_m\left(t\right)\right)}{V(L,t)}\,\mathrm{d}t .
\end{aligned}
\right.
\end{equation}
In order to change the final condition to initial condition, a change of variables can be employed :
\begin{equation*}
\mu(x,t)=\lambda_1(x,\tau-t).
\end{equation*}
Consequently, the adjoint state equation is rewritten as
\begin{equation*}\label{eq:adjoint}
\left\{
\begin{aligned}
& \rho c\frac{\partial \mu(x,t)}{\partial t}
+ \frac{\partial }{\partial x}\left ( k\frac{\partial \mu(x,t)}{\partial x} \right )=S(x,\tau-t),\\
& -k\frac{\partial \mu(x,t)}{\partial{x}}\bigg|_{x=0}=0,
\quad -k\frac{\partial{\mu(x,t)}}{\partial{x}}\bigg|_{x=L}=0,\\
& \mu(x,0)=0.
\end{aligned}
\right.
\end{equation*}
Therefore, according to \eqref{eq:grad1}, the gradients with respect to $q(t)$ and $L$ can be represented by
\begin{equation*}\label{eq:gradient}
\left\{
\begin{aligned}
&\frac{\partial J}{\partial q}(t)=\lambda_1(0,t)=\mu(0,\tau-t),\\
&\frac{\partial J}{\partial L}= \int_{0}^{\tau} \left(\Lambda_L\left(t\right)-\Lambda_m\left(t\right)\right)\frac{\partial \Lambda\left(t\right)}{\partial L}\, \mathrm{d}t
= \int_{0}^{\tau} \frac{2\left(\Lambda_L\left(t\right)-\Lambda_m\left(t\right)\right)}{V(L,t)}\,\mathrm{d}t .
\end{aligned}
\right.
\end{equation*}

Next, we use the conjugate gradient method and the steepest descent method to identify the heat flux $q(t)$ and the thickness $L$, respectively.

\subsection{Update $q$ with the conjugate gradient method }
To numerically reconstruct the heat flux $q$, we shall discretize the heat flux with respect to the time $t$. Suppose that $[0, \tau]$ is discretized as follows
$$0=t_0\leq t_1 \leq \cdots\leq t_i\leq t_{i+1}\leq \cdots \leq t_{N}=\tau.$$
The reconstruction schemes of the heat flux based on the conjugate gradient (CG) method is described as follows
\begin{equation}\label{eq:q}
q_i^{n+1}=q_i^{n}+\beta^{n}p_i^n,
\end{equation}
where the subscript $i$ indicates the discretization of the heat flux in time, and the superscripts $n$ and $n+1$ denote the iteration steps.
$p_i^n$ signifies the conjugate direction and it is generated by the rule
\begin{equation*}
p_i^n=
\begin{cases}
&-g_i^n, \qquad \qquad\quad \ n=1,\medskip \\
&-g_i^{n}+\alpha^{n}p_{i}^{n-1}, \quad n\geq 2,\\
\end{cases}
\end{equation*}
where $\alpha^{n}$ is the CG update parameter given by
\begin{equation*}
\alpha^{n}=\frac{\displaystyle \sum_{i=0}^{N}(g_i^{n})^{T}(g_i^{n}-g_i^{n-1})}{\displaystyle \sum_{i=0}^{N}\left \| g_i^{n-1} \right \|^2},
\end{equation*}
with $\left \| \cdot \right \|$ denoting the Euclidean norm, and
\begin{equation*}
g_i^{n}=\frac{\partial J}{\partial q}\Big|_{t=t_i}^n.
\end{equation*}
In addition, the step size $\beta^n$ is obtained by the exact line search and can be described as
\begin{equation*}
\beta ^{n}=\frac{\displaystyle \sum_{i=0}^{N} \left(\Lambda^n_L\left(t_i\right)-\Lambda_m\left(t_i\right)\right) \Delta t_{in}\,}{ \displaystyle\sum_{i=0}^{N}{\left [ \Delta t_{in} \right ]^2}\,},
\end{equation*}
where $\Lambda^n_L\left(t\right)$ is the solution of the forward problem and $\Delta t_n$ is the change in transmission time and can be expressed as:
\begin{equation*}\label{eq:delta}
\Delta t_{n}(t)=\int_{0}^{L}{\frac{1}{V_{g^{n}}(x,t)}}\mathrm{d}x.
\end{equation*}
Here $V_{g^n}=aT_{g^n}+b$ and $T_{g^n}(x,t)$ is the solution of the following sensitivity equation
\begin{equation*}\label{eq:det}
\left\{
\begin{aligned}
& \rho c\frac{\partial T}{\partial t}=\frac{\partial }{\partial x}\left ( k\frac{\partial T}{\partial x} \right ),\\
& -k\frac{\partial T}{\partial x}\bigg|_{x=0}=g^n,\quad -k\frac{\partial T}{\partial x}\bigg| _{x=L}=0,\\
& T|_{t=0}=T_{0}.
\end{aligned}
\right.
\end{equation*}




\subsection{Update $L$ with the steepest descent method }
The reconstruction of the thickness based on the steepest descent method is described as follows:
\begin{equation}\label{eq:L}
L^{n+1}=L^n+\lambda^nd^n,
\end{equation}
where the superscripts $n$ and $n+1$ denote the iteration steps, and $d^n$ denotes the negative gradient direction respect to $L$,
\begin{equation*}
d^n=- \frac{\partial J}{\partial L}\bigg|_{L=L^n}
=\int_{0}^{\tau} \frac{2\left(\Lambda_m(t)-\Lambda_L(t)\right)}{V(L^n,t)}\,\mathrm{d}t .
\end{equation*}
And the step size $\lambda^n$ is determined by an inexact line search technique called Wolfe-Powell search method. Assuming that $f(L)=J(q,L)$ is differentiable, the Wolfe-Powell search method is used to find $\lambda^n$ along $d^n$ such that
\begin{equation*}
\triangledown f(L^n+\lambda^{n}d^{n})^{T}d^{n}\geq \sigma \triangledown f(L^n)^{T}d^{n},
\end{equation*}
\begin{equation*}
f(L^n+\lambda^{n}d^{n})\leq f(L^n)+\rho \triangledown f(L^n)^{T}d^{n},\quad \rho\in(0,1/2),\sigma\in(\rho,1).
\end{equation*}
Assuming that $\varphi(\lambda^n)=f(L^n+\lambda^nd^n)$, the strategy for computing the step length $\lambda^n$ can be described as follows:\medskip

\noindent\textbf{Step1}: Let $\lambda^0=0,\lambda^{max}>0$, and choose $\lambda^1 \in [\lambda^0,\lambda^{max}], \rho \in (0,1/2), \sigma\in(\rho,1)$. Evaluate $\varphi(\lambda^0)$ and $\varphi^{'}(\lambda^0)$. Let $a_0=\lambda^0,b_0=\lambda^{max},n=0$.\\
\textbf{Step2}: Evaluate $\varphi(\lambda^n)$. If $\varphi (\lambda^{n})\leq\varphi(\lambda^0)+\rho\lambda^{n}\varphi^{'}(\lambda^0),$ go to \textbf{Step3}. Else, go to \textbf{Step4}, set $a_{n+1}=a_{n},b_{n+1}=\lambda^{n}$.\\
\textbf{Step3}: Evaluate $\varphi^{'}(\lambda^n)$. If $\varphi ^{'}(\lambda^{n})\geq \sigma\varphi^{'}(\lambda^0)$, stop. Else, set $a_{n+1}=\lambda^{n},b_{n+1}=b_{n}$, go to \textbf{Step4}.\\
\textbf{Step4}: Let $\lambda^{n+1}=\frac{\displaystyle a_{n+1}+b_{n+1}}{\displaystyle 2 }$, set $n=n+1$, go to \textbf{Step2}.

\subsection{Optimize algorithm iteration format } 

In this paper, we iterate the heat flux $q$ and the thickness $ L$ alternatively, and the proposed algorithm is listed as follows:\medskip

\noindent\textbf{Step1}: Choose an initial point $q_i^0$, $L^0$, $\varepsilon \in (0,1)$.\\
\textbf{Step2}: Fixed $L^n$. Update $q_i$ using the formular \eqref{eq:q}.\\
\textbf{Step3}: Fixed $q_i^{n+1}$ update $L$ using the formular \eqref{eq:L}.\\
\textbf{Step4}: Evaluate $J(q_i^{n+1},L^{n+1})$. If $J(q_i^{n+1},L^{n+1})<\varepsilon$, stop. Else, set $n=n+1$, go to \textbf{Step2}.

\begin{algorithm}[h]
\caption{Alternating iteration algorithm} 
\begin{algorithmic}[1]
\Require
$q^0(N), L^0, crl,n_{max},\varepsilon$.

\Ensure
$q(N)$,
$L$,
$T(nl)$.

\State $q^n\leftarrow q^0$, $L^n\leftarrow L^0$, $J\leftarrow J^0$.
\While {$abs(J)>crl .AND. n<n_{max}$}
\State call gradient
\State $g^n\leftarrow g1$,
\State call cgm
\State $p^n\leftarrow p1$,
\State call bet
\State $\beta^n\leftarrow \beta1$,
\For {$i=1,N$}
\State $q^n(i)\leftarrow q^n(i)-bet*p^n(i) $
\State $g2(i)\leftarrow g1(i)$
\State $p2(i)\leftarrow p1(i)$
\EndFor
\State $//$ update $q$
\If {$abs(L-j1)>\varepsilon$}
\State $j1\leftarrow L$
\State call wolfe(j)
\State $//$ update $L$
\Else \State $aa\leftarrow aa/10$
\State $//$ update regularization parameter
\EndIf

\State compute objective function $J$.

\EndWhile
\end{algorithmic}
\end{algorithm}
%
\section {Convergence Analysis}

In this section, we shall analyze the convergence of the reconstruction scheme proposed in the previous section.
Let $\left( q_{i}^{*},L^{*} \right)$ be the optimal solution to the optimization problem \eqref{eq:lagrangian}, i.e.,
\begin{equation}\label{eq:sql}
J\left(q_{i}^{*},L^{*} \right) \leq J\left(q_{i}^n, L^n \right),\quad \forall\, q_{i}^n \in \mathbb{R}^{N},\ L^n\in \mathbb{R}.
\end{equation}
It is clear that the necessary condition of \eqref{eq:sql} is:
\begin{equation*}
\triangledown J\left( q_{i}^{*},L^{*}\right) = \left( g_{i}^{*}, -d^{*}\right)^{T}=0,
\end{equation*}
and hence it is sufficient for us to prove
\begin{equation}\label{eq:prove}
\lim_{n\rightarrow\infty} \inf \| \left( g_{i}^{n}, d^{n} \right)\|=0.
\end{equation}

Next, we prove that the optimization algorithm that consists of \eqref{eq:q} and \eqref{eq:L} satisfies the convergence condition \eqref{eq:prove}. Before we discuss the convergence, we introduce some notations and important lemmas.

\begin{definition}\label{difinition:d1}
$Polka-Ribi\acute{e}re-Polyak (PRP)$ method is a nonlinear conjugate gradient method, and it has the following form:
\begin{equation}\label{eq:PRP}
\begin{aligned}
&q^{n+1}=q^{n}+\beta^{n}p^{n},\\
&p^{n}=\left\{
\begin{aligned}
&-g^{n} , &\quad n=1, \\
&-g^{n}+\alpha^{n}p^{n-1} , &\quad n\geq2,
\end{aligned}
\right.
\end{aligned}
\end{equation}
where
\begin{equation}\label{eq:alpha2}
\alpha^{n}_{PRP}=\frac{g^{nT}(g^{n}-g^{n-1})}{\left \| g^{n-1} \right \|^2}.
\end{equation}

\end{definition}
\begin{definition}\label{definition:main3}
Exact line search: at each iteration, the step size $\beta^n$ is selected so that
\begin{equation*}
f(q^n+\beta^n p^n)=\underset{\beta}{\min} f(q^n+\beta p^n),
\end{equation*}
where the objective functional is $f(q)=J(q,L), q(t)\in L^{2}\left(\left[0,\tau \right]\right)$.
\end{definition}

\begin{rem}\label{remark2}
Iteration algorithm \eqref{eq:q} is a PRP conjugate method with the exact line search.
\end{rem}

Next, we prove the convergence of the PRP conjugate method with an exact line search as well as its convergence condition. To that end, we first derive several auxiliary lemmas.

\begin{lem}\label{lem:lemma1}
\cite{Powell} Let $\theta_n$ be the angle between the searching direction $p^n$ and the negative gradient direction $-g^n$. Then
\begin{equation*}
\cos \theta_n =\frac{-g^{nT}p^n}{\|g^n\|\|p^n\|}.
\end{equation*}
When the line search is the exact line search, the angle $\theta_n$ is represented by Figure~\ref{fig.ang}. If $\alpha$ is given by \eqref{eq:alpha2}, we have
\begin{equation}\label{eq:theta}
\tan\theta_{n+1}\leq \sec\theta_n \frac{\|g^{n+1}-g^n\|}{\left\|g^n\right\|}.
\end{equation}
\end{lem}

\begin{figure}[h]
\centering
\includegraphics[width=0.3\textwidth]{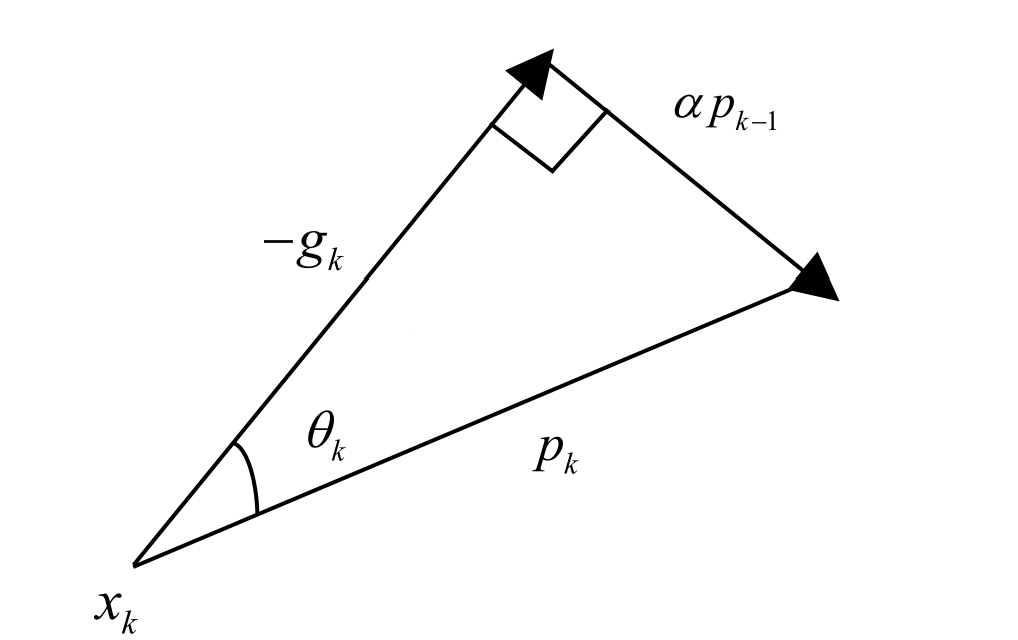}
\caption{The definition of figure}
\label{fig.ang}
\end{figure}
\begin{proof}
Figure~\ref{fig.ang} gives the equation
\begin{equation}\label{eq:f1}
\left\|p^n\right\|=\sec \theta_n\left\|g^n\right\|.
\end{equation}\\
Further, if $n$ is replaced by $n+1$ in Figure~\ref{fig.ang} , we find the identity
\begin{equation}\label{eq:f2}
\alpha^{n+1}\left\|p^{n}\right\|=\tan \theta_{n+1}\|g^{n+1}\|.
\end{equation}
By \eqref{eq:alpha2}, one has
\begin{equation*}
\alpha^{n+1}=\frac{(g^{n+1})^{T}(g^{n+1}-g^{n})}{\left\|g^{n}\right\|^2},
\end{equation*}
and hence by the Cauchy-Schwarz inequality,
\begin{equation}\label{eq:f3}
\alpha^{n+1}\leqslant \frac{\left\|g^{n+1}\right\|\left\|g^{n+1}-g^n\right\|}{\left\|g^{n}\right\|^2}.
\end{equation}
Next, by the elimination of $\left\|g^{n}\right\|$ from \eqref{eq:f1} and\eqref{eq:f2}, one can show the following equality,
\begin{equation}\label{eq:f4}
\alpha^{n+1}=\frac{\tan \theta_{n+1}\left\|g^{n+1}\right\|}{\sec \theta_n\left\|g^n\right\|},
\end{equation}
which in combination with \eqref{eq:f3} and \eqref{eq:f4} readily yields that
\begin{equation*}
\tan\theta_{n+1}\leqslant \sec\theta_n \frac{\left\|g^{n+1}-g^n\right\|}{\left\|g^n\right\|}.
\end{equation*}

The proof is complete.
\end{proof}

\begin{lem}\label{lem:lemma2}
\cite{YYHDYX} Consider the iterative method of the form $q^{n+1}=q^n+\beta^n p^n$, where $p^n$ satisfies the descent condtion $p^{nT}g^n\leq0$,
the step size $\beta^n$ is obtained by the exact line search, the objective functional $f(q)$ is bounded below, and its gradient $\triangledown f(q)$ satisfies the Lipschitz  condition, then
\begin{equation}\label{eq:sum}
\sum_{n\geqslant 1}\frac{(g^{nT}p^n)^2}{\left \| p^n\right \|^2}<\infty,
\end{equation}
and according to the definition of $\theta_n$, \eqref{eq:sum} is equivalent to
\begin{equation}\label{eq:cos}
\sum_{n\geqslant 1}\left\|g^n\right\|^2\cos \theta_n <\infty.
\end{equation}
\end{lem}
\begin{lem}\label{lem:lemma3}
\cite{YYSJ}
If the step size $\beta^n$ is obtained by the exact line search and the objective function $f(q)$ is  uniformly convex, then
\begin{equation*}
f(q^n)-f(q^n+s^n)\geqslant c\left\|s^n\right\|^2
\end{equation*}
holds, where $s^n=q^{n+1}-q^n$, $c>0$ is a constant, and further there is $\left\|s^n\right\|\rightarrow 0$ as $n\rightarrow \infty $.
\end{lem}

\begin{lem}\label{thm:thm1}
Assume that the objective function $f(q)$ is uniformly convex and bounded from below, and its gradient $\triangledown f(q)$ satisfies the Lipschitz condition. Consider the algorithm \eqref{eq:PRP}--\eqref{eq:alpha2}, if the step size $\beta^n$ is obtained by the exact line search,
then
\begin{equation*}
\lim_{n\rightarrow \infty }\inf\left \| g^n \right \|=0.
\end{equation*}
\end{lem}
\begin{proof}
The proof follows a similar spirit to that of Theorem 1 in \cite{Powell}. By an absurdity argument, we assume that the theorem does not hold. Then there is a constant $\gamma $, such that for any $n\geqslant1$,
\begin{equation}\label{eq:b2}
\left\|g^n\right\|\geqslant\gamma.
\end{equation}
By Lemma~\ref{lem:lemma3}, there is $\left\|s^n\right\|\rightarrow 0, n\rightarrow \infty$, which combines with the fact that $\triangledown f(q)$ is Lipschitz continuous implies that there exists a positive integer $m$, such that
\begin{equation}\label{eq:geq}
\left\|g^{n+1}-g^n\right\|\geqslant\frac{1}{2}\gamma,
\end{equation}
which holds for any $n\geqslant m$.
Noticing that for any $\theta_n\in[0,\pi/2)$, there is
\begin{equation*}
\sec\theta_n\leqslant1+\tan\theta_n,
\end{equation*}
which together with \eqref{eq:theta}, \eqref{eq:b2} and \eqref{eq:geq} further implies that
\begin{equation*}
\begin{aligned}
\tan\theta_{n+1}&\leqslant \frac{1}{2}+\frac{1}{4}+\cdots +(\frac{1}{2})^{n-1-m}(1+\tan \theta_{m})
&=1+\tan{\theta_m},\quad \forall n\geqslant m.
\end{aligned}
\end{equation*}
Therefore, the angle $\theta_n$ between the searching direction $p_n$ and the negative gradient direction $-g^n$ is always smaller than an angle bounded above by ${\pi}/{2}$. Therefore, by \eqref{eq:cos}, we have $\left\|g^n\right\| \rightarrow0$, which is a contradiction to \eqref{eq:b2}. Thus the theorem must hold true.

The proof is complete.
\end{proof}

Next, we establish the convergence of the steepest descent algorithm with the Wolfe-Powell step search method.

\begin{lem}\label{lem:lemma4}
\cite{SYYX} Let the objective function $f(L)=J(q,L),L\in \mathbb{R}$ be differentiable and bounded from below, and $g(L)=\triangledown f(L)$ satisfy the Lipschitz condition. Suppose that the steepest descent method generates a sequence $L^n, n\geq1$, using the recurrence
\begin{equation*}
L^{n+1}=L^{n}+\lambda^nd^n,
\end{equation*}
where the direction $d^n$ denotes the negative gradient direction. If the step size $\lambda$ satisfies
\begin{equation*}
\left\{
\begin{aligned}
&(g^{n+1})^{T}d^{n}\geq \sigma (g^n)^{T} d^{n},\\
&f(L^{n+1})\leq f(L^n)-\rho \lambda^{n}(d^{n})^2,\quad \rho\in(0,1/2),\sigma\in(\rho,1),
\end{aligned}
\right.
\end{equation*}
where $g^{n+1}=g(L^{n+1})=g(L^n+\lambda^nd^n)$,
then one has
\begin{equation*}
f(L^n)-f(L^{n+1})\geq\frac{\rho(1-\sigma )}{M}\left\|g^n\right\|^2\cos^2\theta_n.
\end{equation*}
\end{lem}

\begin{lem}\label{thm:thm2}
If the function $f(L)$ is continuously differentiable and satisfys the conditions of Lemma~\ref{lem:lemma4}, then sequence $L^n$ generated by the steepest descent method satisfys:
\begin{equation*}
\lim_{n\rightarrow\infty }\left\|g^n\right\|^2=0.
\end{equation*}
\end{lem}
\begin{proof}
By Lemma~\ref{lem:lemma4},
\begin{equation*}
f(L^n)-f(L^{n+1})\geq\frac{\rho(1-\sigma )}{M}\left\|g^n\right\|^2\cos^2\theta_n.
\end{equation*}
Thus
\begin{equation}\label{eq:aa1}
\begin{aligned}
&f(L^0)-f(L^{n+1})
=[f(L^0-f(L^1)]+[f(L^1)+f(L^2)]+\cdots +[f(L^{n})-f(L^{n+1})]\\
&=\sum_{k=0}^{n}[f(L^{k})-f(L^{k+1})]\geq\sum_{k=0}^{n}\frac{\rho(1-\sigma )}{M}\left\|g^k\right\|^2\cos^2(d^k,-g^k).
\end{aligned}
\end{equation}
Notice that $\cos^2(d_k,-g_k)=1,k=1,2,\cdots,{n}$. We therefore have from \eqref{eq:aa1} that
\begin{equation}\label{eq:aa2}
\sum_{k=0}^{n}\frac{\rho(1-\sigma )}{M}\left\|g_k\right\|^2\leq f(L^0)-f(L^{n+1}).
\end{equation}
Since $f(L)$ is bounded from below, one sees that $f(L^0)-f(L^{n+1})< \infty$, which together with \eqref{eq:aa2} readily implies that
\begin{equation*}
\lim_{n\rightarrow \infty}\left\|g^n\right\|^2=0.
\end{equation*}

The proof is complete.
\end{proof}

Next we prove that the optimization algorithm consisting of \eqref{eq:q} and \eqref{eq:L} satisfies the convergence condition \eqref{eq:prove}.
\begin{thm}
Consider the iterative algorithm consisting of \eqref{eq:q} and \eqref{eq:L}:
\begin{equation*}
\left\{
\begin{aligned}
&q_i^{n+1}=q_i^n+\beta^np^n_i,\\
&L^{n+1}=L^n+\lambda^nd^n,
\end{aligned}
\right.
\end{equation*}
Assume that the objective functional $J(q,L)$ satisfies the following conditions:\\
(a): $J(q,L)$ is continuously differentiable with respect to $q(t)$ and $L$;\\
(b): $J(q,L)$ is uniformly convex;\\
(c): $J(q,L)$ is bounded from below;\\
(d): Its gradient $\triangledown J(q,L)$ is Lipschitz continuous.\\
Then the optimization algorithm consisting of \eqref{eq:q} and \eqref{eq:L} satisfies the convergence condition \eqref{eq:prove}, i.e.,
\begin{equation*}
\liminf_{n\rightarrow \infty } \|\left(g_i^n ,d^n\right)\|_1=0,
\end{equation*}
{where $\|\cdot \|_1$ denotes the $l^1$-norm, namely $\|\left(g_i^n ,d^n\right)\|_1=\| g_i^n \|_1+|d^n|.$}
\end{thm}
\begin{proof}
We first consider the iterative algorithm \eqref{eq:q}.
Since the step size $\lambda^{n}$ is searched by the Powell-Wolfe method, it satisfies the following conditions:
\begin{equation}\label{eq:e}
\begin{aligned}
&\varphi (\lambda^{n})\leq\varphi(0)+\rho\lambda^{n}\varphi^{'}(0),\\
&\varphi ^{'}(\lambda^{n})\geq \sigma\varphi^{'}(0),\quad \rho\in(0,1/2),\sigma\in(\rho,1),
\end{aligned}
\end{equation}
where $\varphi (\lambda^{n})=J(q^n, L^n+\lambda^{n}d^{n})$. By using conditions (a) (c) (d) and \eqref{eq:e}, we see that Lemma~\ref{thm:thm2} holds.
Hence by Lemma~\ref{thm:thm2}, we have
\begin{equation}\label{eq:first}
\lim_{n\rightarrow\infty}\left | d^n \right |=0.
\end{equation}

We proceed to consider the iterative algorithm \eqref{eq:L}. By virtue of the conditions (a) (b) (c) (d), we see that the iteration algorithm \eqref{eq:q} is a PRP conjugate method with the exact line search. By Lemma~\ref{thm:thm1},
\begin{equation}\label{eq:second}
\lim_{n\rightarrow \infty }\inf\left \| g_i^n \right \|=0,
\end{equation}
According to \eqref{eq:first} and \eqref{eq:second},
\begin{equation*}
\lim_{n\rightarrow \infty }\inf \left(\| g_i^n \|+|d^n|\right)=0.
\end{equation*}
Thus,
\begin{equation*}
\lim_{n\rightarrow \infty }\inf \|\left(g_i^n ,d^n\right)\|_1
=\lim_{n\rightarrow \infty }\inf \left(\| g_i^n \|+|d^n|\right)=0.
\end{equation*}

The proof is complete.
\end{proof}

\section {Numerical Examples}

In this section, we present several numerical examples to verify the effectiveness and robustness of the proposed scheme in simultaneously reconstructing the surface heat flux and the thickness of a solid structure under different {acoustic time accuracies, initial fluxes, and initial thicknesses.} It is emphasized that all the data in our numerical examples are collected by lab experiments following the setup described in Figure~\ref{fig.Model}.

The specimens with a thickness of $L=50\, \rm{mm}$ are heated at one boundary and the surface heat flux is $q(t)=10^5\, \rm{J/s}$. The ultrasonic wave probes are stalled on the other boundary of the specimens with the detection frequency set to be $\omega=1\, \mathrm{Hz}$, and the total detection time set to be $\tau=500\,\rm{s}$. Moreover, the thermal conductivity of the specimens is $k=50\,\rm{W/(m\cdot ^\circ\!C)}$, specific heat is $c=400\, \rm{J/(kg\cdot^\circ\! C)}$ and density of the material is $\rho=7800\, \rm{kg/m^3}$. The initial tempareture is chosen as $T_0=\rm{26\, ^\circ C}$. The relationship between the velocity and temperature is given as follows:
\begin{equation*}
V(T)=-0.4521T+3259.9.
\end{equation*}
In the following numerical examples, the stopping criterion for the iterations is set to be $J(q, L)<5\times10^{-18}$. The Fortran software is used for implementing of Algorithm 1.

The reconstruction results of thickness under the acoustic time accuracy of $10^{-9}$, $10^{-10}$ and $10^{-11}$ are respectively shown in the Table~\ref{tab.1}.
\begin{table}[H]
\centering
\begin{tabular}{llllll}
\hline
Acoustic time&initial heat flux&initial thickness&reconstructed&iterations\\
accuracy $(s)$&$q^0(\rm{{J}/{s}})$&$L^0(\rm{mm})$&thickness $L(\rm{mm})$&$\rm{n}$\\
\hline
$10^{-9}$&0&3&50.0006&138\\
$10^{-10}$&$0$&3&50.0006&64\\
$10^{-11}$&$0$&3&50.0006&63\\
$10^{-9}$&$1\times10^3$&45&50.0032&208\\
$10^{-10}$&$1\times10^3$&45&50.0028&115\\
$10^{-11}$&$1\times10^3$&45&50.0025&108\\
\hline
\end{tabular}
\caption{Convergence of the iteration method with different initial guesses and measurement errors.}
\label{tab.1}
\end{table}
It can be found that the thickness can be reconstructed effectively under different acoustics time accuracies.
The error and the number of iterations show that under the same initial value, the accuracy of acoustic time will affect the convergence speed of the algorithm. If the accuracy of acoustic time reaches $10^{-10}$ or $10^{-11}$, one can achieve much accurate reconstruction results. Thus, in the following numerical examples, we adopt the measurement data with an acoustic time accuracy of $10^{-10}$ or $10^{-11}$ to study the effect of the initial values on the inversion procedure.

Figure \ref{fig.t1} presents the reconstruction results of the heat flux under different initial thicknesses and different initial heat flux conditions. The acoustic accuracy is fixed to be $10^{-10}$.

\begin{figure}[H]
\hfill\subfigure[$ $]{\includegraphics[width=0.48\textwidth]
{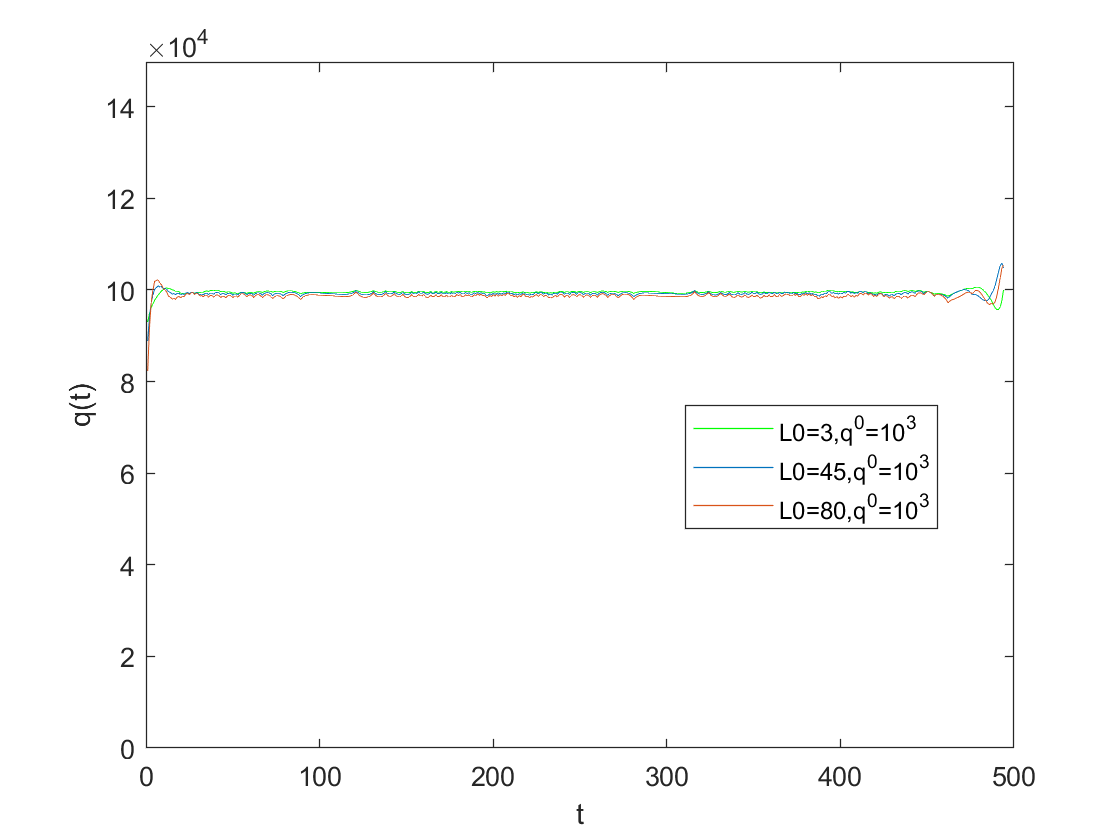}}\hfill
\hfill\subfigure[$ $]{\includegraphics[width=0.48\textwidth]
{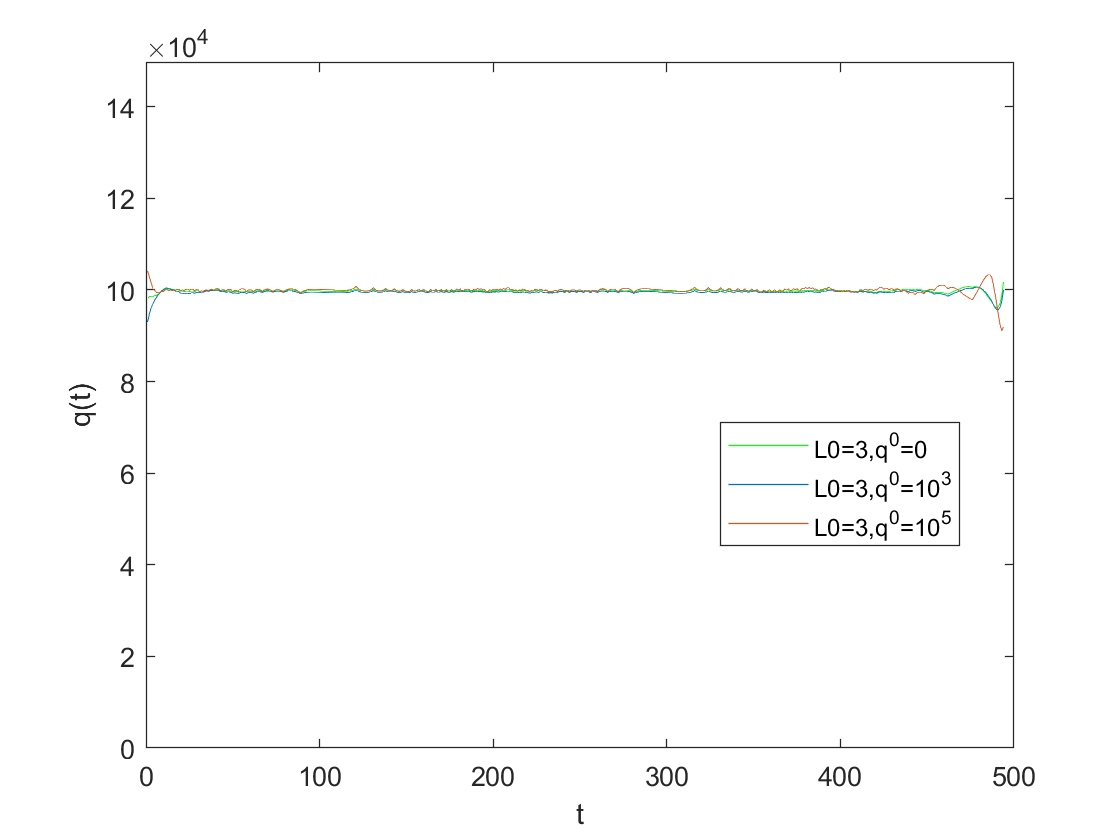}}\hfill
\caption{ \label{fig.t1} Heat flux inversion results with the acoustic time accuracy fixed to be $10^{-10}$. (a): under a fixed initial heat flux and different initial thicknesses, (b): under a fixed initial thickness and different initial heat fluxes.}
\end{figure}

Figure \ref{fig.t3} shows the reconstruction results of the heat flux under different initial thicknesses and different initial heat flux conditions. The acoustic accuracy is fixed to be $10^{-11}$.

\begin{figure}[H]
\hfill\subfigure[$ $]{\includegraphics[width=0.48\textwidth]
{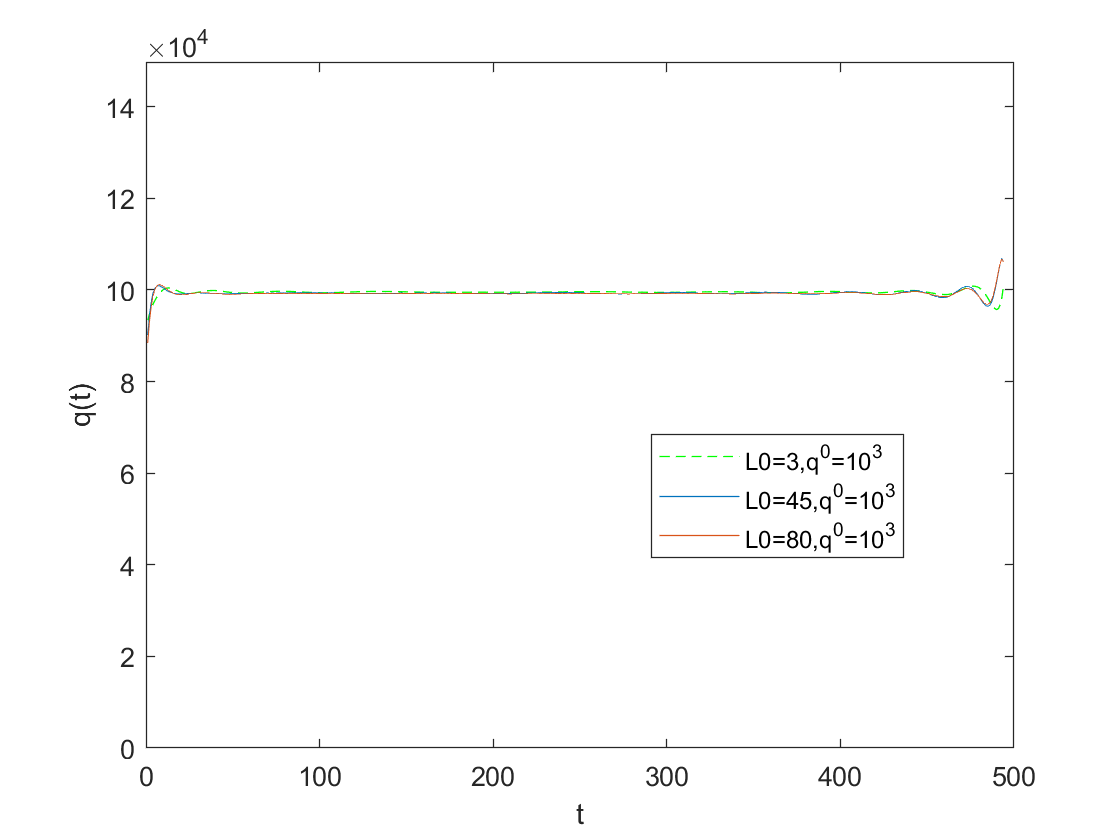}}\hfill
\hfill\subfigure[$ $]{\includegraphics[width=0.48\textwidth]
{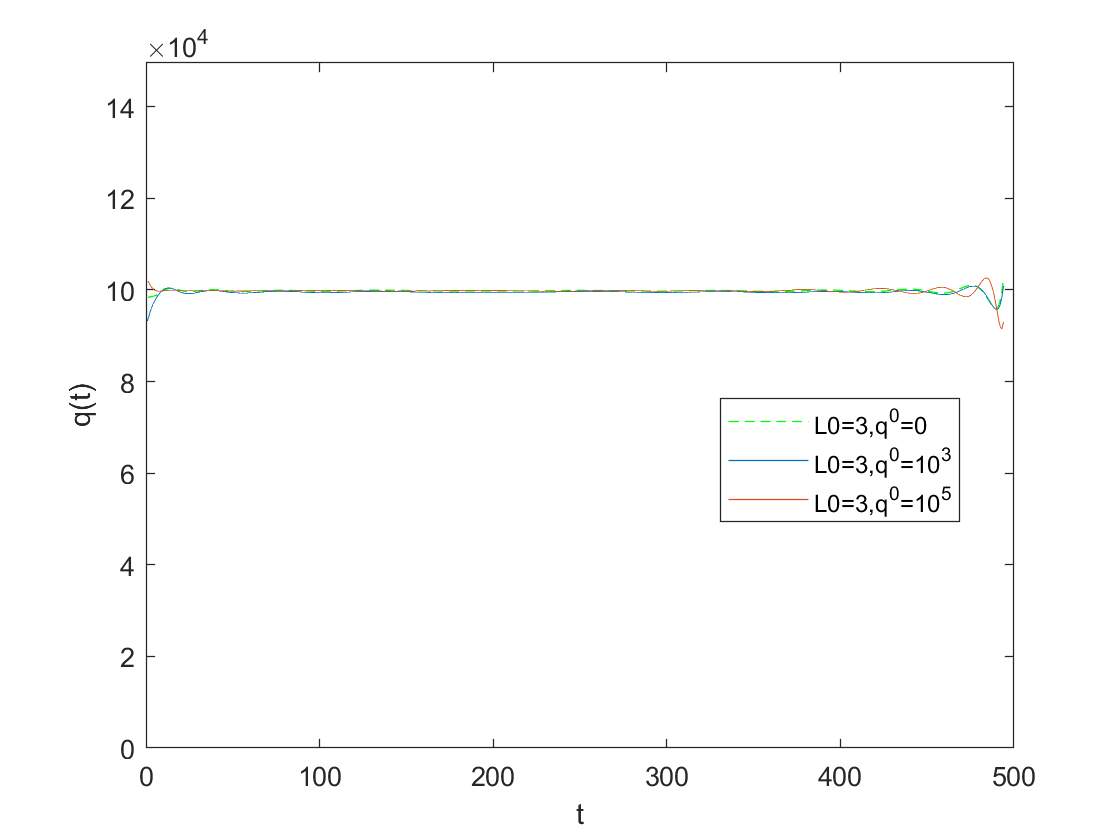}}\hfill
\caption{ \label{fig.t3} Heat flux inversion results with the acoustic time accuracy fixed to be $10^{-11}$. (a): under a fixed initial heat flux and different initial thicknesses, (b): under a fixed initial thickness and different initial heat fluxes.}
\end{figure}

By observing the surface heat flux reconstruction results in Figures \ref{fig.t1} and \ref{fig.t3}, it can be found that when the acoustic time accuracy is $10^{-10}$ and $10^{-11}$, under different initial conditions, the inversion value of the heat flux converges to the real value. Moreover, when the acoustic time accuracy reaches  $10^{-11}$, the inversion value of heat flux very close to the real value, which achieves a much accurate reconstruction.

Table~\ref{tab.2} lists the inversion results and the iteration times of the thickness under different initial thicknesses and different initial heat fluxes initial conditions. The acoustic time accuracy is $10^{-10}$ or $10^{-11}$.

\begin{table}[h]
\centering
\begin{tabular}{lllll}
\hline
acoustic time &initial heat flux&initial thickness&reconstructed thickness&iterations\\
accuracy&$q^0(\rm{{J}/{s}})$&$L^0(\rm{mm})$&$L(\rm{mm})$&$n$\\
\hline
$10^{-10}$&0&3&50.0006&64\\
$10^{-10}$&$1\times10^3$&3&50.0016&58\\
$10^{-10}$&$1\times10^5$&3&50.0000&133\\
$10^{-10}$&0&45&50.0025&126\\
$10^{-10}$&$1\times10^3$&45&50.0028&115\\
$10^{-10}$&0&80&50.0028&77\\
$10^{-10}$&$1\times10^3$&80&50.0046&92\\
$10^{-11}$&0&3&50.0006&63\\
$10^{-11}$&$1\times10^3$&3&50.0016&57\\
$10^{-11}$&$1\times10^5$&3&50.0006&147\\
$10^{-11}$&0&45&50.0025&126\\
$10^{-11}$&$1\times10^3$&45&50.0025&108\\
$10^{-11}$&0&80&50.0028&77\\
$10^{-11}$&$1\times10^3$&80&50.0029&80\\
\hline
\end{tabular}
\caption{Convergence of the proposed iteration method with different initial guesses and measurement errors.}
\label{tab.2}
\end{table}

The results show that the iterative algorithm converges very fast and robust with different initial conditions.
\section{Conclusion}

Based on the ultrasonic echo method and the inverse problem analysis method of the heat conduction, combined with the optimization model, a method of simultaneously reconstructing the thickness and the surface heat flux of a solid structure is established in this paper. This approach is non-destructive and non-contact and it can be used to recover the surface heat flux and the wall thickness at the same time. It possesses a high engineering value. We provide a rigorous convergence analysis of the proposed numerical scheme. By using experimental lab data, we conducted extensive numerical experiments to verify the effectiveness and efficiency of the newly developed method. It is found that as long as the acoustic time accuracy reaches $10^{-10}$ or $10^{-11}$, the proposed iteration method converges very fast and robust with respect to different initial guesses.

\section*{Acknowledgement}

 The work of H Liu was supported by a startup fund from City University of Hong Kong and the Hong Kong RGC General Research Fund (projects 12301420, 12302919, 12301218).

%
%
%
%
%
%
%
%
%
%

\end{document}